\documentclass[12pt]{amsart}
\usepackage[utf8]{inputenc}
\usepackage{amsmath,amssymb,amsthm}

\usepackage{caption} 
\usepackage[en-AU]{datetime2}

\usepackage{graphicx}
\usepackage{hyperref}
\usepackage[dvipsnames]{xcolor}
\usepackage{tikz}
    \usetikzlibrary{calc}
\usepackage{cleveref}
\usepackage{listings}
\newtheorem{theorem}{Theorem}[section]
\newtheorem{observation}[theorem]{Observation}

\newtheorem{proposition}[theorem]{Proposition}

\newtheorem*{thm*}{Theorem}

\theoremstyle{remark}

\theoremstyle{definition}
\newtheorem{definition}[theorem]{Definition}

\newtheorem{problem}{Problem}

\usepackage{calc}
\usepackage[shortlabels]{enumitem}
    \newlength{\circlabelwidth}
    \setlist{nosep,before={\parskip=0pt plus 2pt},after={\parskip=0.5em plus 2pt}}
    \setlist[enumerate]{label=\textup{\arabic*.}}
    \newlist{subprob}{enumerate}{2}
        \setlist[subprob,1]{label={(\roman*)}}
        \setlist[subprob,2]{label={(\arabic*)}}
    \setlist[itemize]{labelindent=10pt,labelwidth=\circlabelwidth,leftmargin=!,label=$\circ$}
    \newlist{problems}{enumerate}{3}
        \setlist[problems,1]{before=\setupstar,label=\textup{\arabic*.}, itemsep=2pt, topsep=8pt,ref=\textup{\arabic*}}
        \setlist[problems,2]{before=\setupstar,label=(\alph*),parsep=0pt}
        \setlist[problems,3]{before=\setupstar,label=(\roman*),parsep=0pt}

\DeclareMathOperator{\conv}{conv}
\newcommand*{\R}{\mathbb{R}}
\newcommand*{\Z}{\mathbb{Z}}

\title{Hollow polytopes with many vertices}
\author{Srinivas Arun}
\author{Travis Dillon}
\address{Srinivas Arun, Massachusetts Institute of Technology, Cambridge, MA 02139, USA}
\email{sarun@mit.edu}
\address{Travis Dillon, Massachusetts Institute of Technology, Cambridge, MA 02139, USA}
\email{travis.dillon@mit.edu}

\begin{document}

\begin{abstract}
Given a set $S \subseteq \R^d$, a \emph{hollow polytope} has vertices in $S$ but contains no other point of $S$ in its interior. We prove upper and lower bounds on the maximum number of vertices of hollow polytopes whose facets are simplices or whose vertices are in general position. We also obtain relatively tight asymptotic bounds for polytopes which do not contain lattice segments of large length.
\end{abstract}

\maketitle

\section{Introduction}

A subset $T$ of a point set $S \subseteq \R^d$ is called \emph{empty in $S$} if every point in $T$ is a vertex of $\conv(T)$ and no other point of $S$ is contained in $\conv(T)$; in short, $\conv(T)\cap S =\operatorname{vert}\!\big(\conv(T)\!\big)$. Empty polytopes in a discrete set $S$ are strongly connected to the intersection properties of that set with convex bodies:

\begin{thm*}[Hoffman \cite{Hoffman1979}]
    Given a set $S \subseteq \R^d$, let $h(S)$ be the minimal positive integer, if any exists, such that the following \emph{Helly-type theorem} holds:
    \begin{quote}
    For any finite family $\mathcal{F}$  of convex sets in $\mathbb{R}^d$, if every $h(S)$ or fewer sets in $\mathcal{F}$ contain a point of $S$ in their intersection, then $\bigcap \mathcal F$ contains a point in $S.$
    \end{quote}
    Then $h(S)$ equals the maximum number of vertices of an empty subset of $S.$
\end{thm*}

For instance, any empty subset of $\mathbb{Z}^d$ has at most $2^d$ vertices: If $T \subseteq \Z^d$ has at least $2^d + 1$ points, then two of them, say $x$ and $y$, have the same parity in every coordinate, so $\frac{1}{2}(x+y)$ is an integer point in $\conv(T)$ that is not a vertex. On the other hand, $T = \{0,1\}^d$ is empty in $\Z^d$. So $h(\Z^d) = 2^d$, which proves the following theorem, (first proved by Doignon \cite{Doignon} and later, independently, by Bell \cite{Bell1977} and Scarf \cite{Scarf1977}):

\begin{thm*}
    Let $\mathcal F$ be a finite family of convex sets in $\R^d$. If the intersection of any $2^d$ or fewer sets in $\mathcal F$ contains an integer point, then $\bigcap \mathcal F$ does, too.
\end{thm*}

A subset $T$ of a point set $S\subseteq\mathbb{R}^n$ is called \emph{hollow in $S$} if the \emph{interior} of $\conv(T)$ contains no points of $S$. There is a large literature on so-called \emph{hollow} polytopes, much of which is focused on classification of hollow polytopes up to unimodular equivalence. There is also a strong connection between hollow polytopes and the celebrated Flatness Theorem (stated as  \Cref{thm:flatness} in this paper). See \cite{iglesias2019hollow} for a recent survey of these aspects.

In this paper, we seek bounds on the number of vertices of hollow polytopes, in essence importing a central question about empty polytopes to the world of hollow polytopes.

Since many of our results are asymptotic, we will use standard order of magnitude notation: for two functions $f(k,d)$ and $g(k,d)$, we write $f = O(g)$ to denote that there is an absolute constant $c > 0$ such that $f(k,d) \leq c\, g(k,d)$ for all $k$ and $d$. Similarly, the notation $f = O_d(g)$ indicates that there is a positive function $c(d)$ such that $f(k,d)\leq c(d)\, g(k,d)$ for all $k$ and $d.$. We write $f = \Theta_d(g)$ if $f = O_d(g)$ and $g = O_d(f)$. We will also say that a set in $\mathbb{R}^d$ is in ``general position'' if no $d+1$ points of the set are contained in a hyperplane.

The main results of the paper are as follows.

\begin{theorem}\label{hollowthm}
    If $P$ is a hollow polytope in $\Z^d$ in which every facet is a simplex, then $P$ has at most $2^d$ vertices. Furthermore, if the vertices of $P$ are in general position, then $P$ has $O\big( d^2 (\log d)^3 \big)$ vertices.
\end{theorem}

As we describe in \Cref{sec:interior-lattice}, to obtain any bound on the number of vertices, some restriction (such as simplicial facets or vertices general position) must be imposed; these two conditions seem fairly reasonable.

In \Cref{sec:lattice-segments}, we relax the emptiness condition in a different way and prove bounds on the number of vertices of polytopes that contain no ``lattice segments'' of large length.

\begin{theorem}\label{segthm}
    If $P$ is a convex lattice polytope in $\R^d$ which does not contain a set of the form $\{x, x+y, x+2y, \dots, x+ky\}$ with $x,y \in \Z^d$ and $y\neq 0$, then $P$ has $O_d(k^{d-1})$ vertices.
\end{theorem}

We also construct a polytope with $\Theta_d(k^{(d-1)\frac{d}{d+1}})$ vertices containing no such set.

We conclude in \Cref{sec:problems} with several open problems.

\section{Interior lattice points}\label{sec:interior-lattice}

How many points of $\Z^d$ can a hollow polytope contain? As it turns out, arbitrarily many: Take, for example, a lattice rectangle $[0,1] \times [0,n]$. Even the number of vertices can be arbitrarily large:

\begin{observation}
    For any $d \geq 3$, there is convex polytope $\mathcal{P}$ in $\mathbb{Z}^d$ with arbitrarily many vertices such that the only lattice points in $\mathcal{P}$ are in its vertices, edges, or two-dimensional faces.
\end{observation}
\begin{proof}
    Let $x_1,\dots,x_n$ be the vertices of a convex $n$-gon in $\Z^2 \subset \Z^d$, and let $e_i$ be the $i$th standard basis vector. Then $\conv(x_1,\dots,x_n, e_3,\dots,e_d)$ is a polytope with $n+d-2$ vertices.
\end{proof}

In order to obtain meaningful upper bounds, we therefore must impose some further restrictions. To that end, let $\operatorname{hol}_{\textup{gp}}(\Z^d)$ denote the maximum number of vertices of a hollow lattice polytope whose vertices are in general position. It is not obvious that $\operatorname{hol}_{\textup{gp}}(\Z^d)$ is finite, either; while the general position condition is restrictive, we still have the freedom of allowing many integer points on the boundary of the convex hull.

In fact, $\operatorname{hol}_{\textup{gp}}(\Z^d)$ is finite, and one crucial fact implying this is that a polytope whose vertex set is in general position has simplices as facets. Therefore, we also define $\operatorname{hol}_\triangle(\Z^d)$ as the maximum number of vertices in a hollow simplicial lattice polytope. Since any hollow lattice polytope with vertices in general position is simplicial, $\operatorname{hol}_{\textup{gp}}(\Z^d) \leq \operatorname{hol}_\triangle(\Z^d)$. Our first result is that $\operatorname{hol}_\triangle(\Z^d)$ is finite.

\begin{theorem}\label{edgeDoi}
    $\operatorname{hol}_\triangle(\Z^d)\leq 2^d$.
\end{theorem}
\begin{proof}
    We will show that for any hollow simplicial lattice polytope, we can construct an empty lattice polytope with the same number of vertices.

    Assume that $P$ is a simplicial lattice polytope whose boundary contains a lattice point $x$ that is not a vertex. Let $y$ be a vertex of the minimal-dimension face of $P$ that contains $x$, and let $Q$ be the polytope whose vertex set is $\{x\}\cup \operatorname{vert}(P)\setminus\{y\}$. Then $Q$ is a simplicial lattice polytope with strictly fewer non-vertex lattice points on its boundary and the same number of vertices as $P$. By induction on $|P \cap \Z^d|$, there is an empty polytope with the same number of vertices as $P$, so Doignon's theorem implies that $\lvert\operatorname{vert}(P)\rvert\leq 2^d$.
\end{proof}

The bound in Theorem~\ref{edgeDoi} is not necessarily sharp for $d > 2$, but we unfortunately do not have a better bound for simplicial polytopes. For polytopes whose vertices are in general position, however, we can obtain a much better bound using the concept of lattice width.

\begin{definition}
    The \emph{width} of a convex body $K \subseteq \R^d$ in the direction of a vector $v \in \R^d$ is 
    \[
        w_v(K) = \sup_{x \in K} \langle x,v\rangle - \inf_{x \in K} \langle x,v\rangle.
    \]
    The \emph{lattice width} of a convex body $K$ is 
    \[
        w_L(K) = \min_{v \in \Z^d \setminus \{0\}} w_v(K).
    \]
\end{definition}

The \emph{lattice hyperplanes} orthogonal to an integer vector $v$ are those hyperplanes orthogonal to $v$ that intersect $\Z^d$. Two adjacent lattice hyperplanes orthogonal to $v$ are distance $1/\|v\|$ apart. So if $w_v(K) = \alpha$, at most $\alpha+1$ lattice hyperplanes orthogonal to $v$ intersect $k$, and if $w_L(K) = \alpha$, then there is a lattice vector $v$ such that at most $\alpha + 1$ lattice hyperplanes orthogonal to $v$ intersect $K$.

In 1948, Khinchine \cite{khinchine} proved the fundamental fact that any convex body containing no lattice points has bounded lattice width. The minimal upper bound on this lattice width is called the \emph{flatness constant}. A series of works improved Khinchine's upper bound; Kannan and Lov\'asz in 1988 provided an elegant proof that the flatness constant in $\R^d$ is at most $d^2$ \cite{kannan-lovasz-covering}. On the other hand, the flatness constant is at least $d$: The interior of the lattice simplex $\conv(0, d e_1, \dots, d e_d)$ is a convex set with lattice width $d$ containing no lattice points. Recently, Reis and Rothvoss obtained an upper bound that is near optimal:

\begin{theorem}[Flatness theorem, Reis--Rothvoss \cite{flatness-thm}]\label{thm:flatness}
    If $K$ is a convex body in $\R^d$ that contains no lattice points, $w_L(K) = O\big(d (\log d)^3\big)$.
\end{theorem}

We can apply this result to obtain an improved upper bound for $\operatorname{hol}_{\textup{gp}}(\Z^d)$.

\begin{theorem}\label{holgp}
    $\operatorname{hol}_{\textup{gp}}(\Z^d) = O\big(d^2(\log d)^3\big)$.
\end{theorem}
\begin{proof}
    Let $X$ be a set of integer points in $\Z^d$ in general position such that $\conv(X)$ is hollow and $\operatorname{vert}\!\big(\!\conv(X)\big) \cap \Z^d = X$. Let $c \in \conv(X)$. Then $K:= \frac{1}{2}(\conv(X)-c) + c$ is a convex body that contains no lattice points. Using \Cref{thm:flatness}, we obtain a vector $v$ such that $K$ intersects at most $O(d(\log d)^3)$ lattice hyperplanes orthogonal to $v$; therefore $\conv(X)$ intersects at most $O(d (\log d)^3)$ lattice hyperplanes orthogonal to $v$, as well. Each lattice hyperplane contains at most $d$ points of $X$, so $X$ can contain no more than $O(d^2 (\log d)^3)$ points.
\end{proof}
Theorem~\ref{edgeDoi} and Theorem~\ref{holgp} together imply Theorem~\ref{hollowthm}.

As for lower bounds, the lattice simplex $\conv(0,e_1,\dots,e_d)$ shows that $\operatorname{hol}_\triangle(\Z^d) \geq \operatorname{hol}_{\textup{gp}}(\Z^d) \geq d+1$. A simple construction slightly improves this lower bound.

\begin{proposition}
    $\operatorname{hol}_\triangle(\Z^d)\geq 2d$.
\end{proposition}
\begin{proof}
    We utilize an example from \cite[p. 11]{0-1-polytopes}. Let $\mathbf 1$ be the vector with $1$ in each coordinate. The set $X = \{e_i\}_{i=1}^d \cup \{\mathbf 1 - e_i\}_{i=1}^d$ lies in the unit cube $\{0,1\}^d$, so $\conv(X)$ is hollow. Moreover, $X$ is symmetric about the point $\frac{1}{2}\cdot \mathbf 1$, so there is an affine transformation taking $X$ to the standard cross-polytope; therefore $\conv(X)$ is simplicial.
\end{proof}





\section{Lattice segments}\label{sec:lattice-segments}

We now investigate polytopes which avoid certain point sets in their convex hull. The simplest possible example of this is prohibiting segments of a given length.

\begin{definition}
    We use $[x,y]$ to denote the line segment with endpoints $x$ and $y$. If $x,y \in \Z^d$ with $\big\lvert[x,y]\cap \Z^d\big\rvert = k$, we say that $[x,y]$ is a \emph{lattice segment of length} $k-1$.
\end{definition}

Notice that a single point has length $0$. We first consider the maximum number of lattice points a polytope can contain without containing a lattice segment of length $k$.

\begin{proposition}
    If $P$ is a convex lattice polytope in $\R^d$ and $P$ does not contain a lattice segment of length $k$, then $|P\cap \Z^d| \leq k^d$.
\end{proposition}
\begin{proof}
    If $|P \cap \Z^d| > k^d + 1$, then there are two lattice points $x,y \in P \cap \Z^d$ whose coordinates are equal modulo $k$. If we set $z = \frac{1}{k}(y-x) \in \Z^d$, then $x, x+z, x+2z \dots, x+k z = y$ is a set of $k+1$ colinear lattice points contained inside $P$; so $P$ contains a lattice segment of length at least $k$.
\end{proof}

Since the hypercube $[1,k]^d$ contains exactly $k^d$ lattice points but no set of $k+1$ colinear lattice points, this inequality is tight. On the other hand, such lattice polytopes contain many fewer vertices. (This is \Cref{segthm}.)

\begin{proposition}
    Any convex lattice polytope in $\R^d$ that does not contain a lattice segment of length $k$ has $O_d(k^{d-1})$ vertices.
\end{proposition}
\begin{proof}
    We proceed by induction. The theorem is true if $d=1$, since every polytope in $\R^1$ has at most 2 vertices. Suppose that $P$ is a convex lattice polytope in $\R^d$ with no lattice segment of length $k$. If $|P \cap k\Z^d| \geq 2$, then $P$ contains a lattice segment of length $k$. Thus $|P \cap k\Z^d| \leq 1$. By translating $P$ by an integer vector, we may assume that $P \cap (2k)\Z^d = \emptyset$.

    In this case, $\frac{1}{2k}P$ does not intersect the integer lattice. By the \hyperref[thm:flatness]{Flatness Theorem}, there is a constant $c_d$ so that $\frac{1}{2k}P$ has lattice width at most $c_d$; so $P$ has lattice width at most $2c_d k$. This means that there is an integer vector $v$ such that at most $2c_d k+1$ lattice hyperplanes orthogonal to $v$ intersect $P$. Every vertex of $P$ lies on one of these hyperplanes, and each hyperplane has at most $O_d(k^{d-2})$ vertices by induction. So $P$ has at most $O_d( k^{d-1} )$ vertices.
\end{proof}

On the other hand, there is a lattice polytope with almost the same number of vertices and no lattice segment of length $k$.

\begin{proposition}
    There is a convex lattice polytope in $\R^d$ with $\Theta_d(k^{(d-1) \frac{d}{d+1}})$ vertices and no lattice segment of length $k$.
\end{proposition}
\begin{proof}
    Take $P = \conv\big(\frac{k-1}{2}B^d \cap \Z^d\big)$, where $B^d$ is the unit ball in $\R^d$. Since $P \subset \frac{k-1}{2}B^d$, it does not contain a lattice segment of length $k$. On the other hand, B\'ar\'any and Larman \cite{barany_integer-ball} proved that $P$ has $\Theta_d(k^{(d-1) \frac{d}{d+1}})$ vertices.
\end{proof}

\section{Open Problems}\label{sec:problems}
Although we obtain lower and upper bounds for $\operatorname{hol}_\triangle(\Z^d)$, they are far apart asymptotically: We only know that $2d \leq \operatorname{hol}_\triangle(\Z^d) \leq 2^d$. Is the growth rate of $\operatorname{hol}_\triangle(\Z^d)$ truly exponential, or is it much slower? There is an asymptotic difference in the bounds for $\operatorname{hol}_{\textup{gp}}(\Z^d)$, as well, though the gap is less dramatic.

\begin{problem}
    Improve the asymptotic bounds for $\operatorname{hol}_{\textup{gp}}(\Z^d)$ and $\operatorname{hol}_\triangle(\Z^d)$.
\end{problem}

Also, our bounds for the number of vertices of a lattice polytope in $\R^d$ that does not contain a lattice segment of length $k$, although close, do not quite match. We suspect that our lower or upper bound is tight, but make no conjecture which it is.

\begin{problem}
    Establish tight asymptotic bounds on the maximum number of vertices of a lattice polytope that does not contain a lattice segment of length $k$.
\end{problem}

And, of course, ``lattice segment of length $k$'' can be replaced by any family of point configurations, for a whole constellation of related problems.

\section{Acknowledgements}
This research was conducted under the auspices of the \scalebox{0.9}{MIT} \scalebox{0.9}{PRIMES}--\scalebox{0.9}{USA} program. Dillon was further supported by a National Science Foundation Graduate Research Fellowship under Grant No. 2141064.

\bibliographystyle{amsplain-nodash}
\bibliography{ref}

\end{document}